\documentclass[11pt]{article}
\usepackage{amsmath, amssymb, latexsym}
\usepackage{amsthm} 
\usepackage{amscd}
\usepackage{amsfonts}

\newcommand{\ka}{{\mathfrak a}}
\newcommand{\f}{{\mathfrak f}}
\newcommand{\g}{{\mathfrak g}}
\newcommand{\h}{{\mathfrak h}}

\newcommand{\s}{{\mathfrak s}}

\newcommand{\n}{{\mathfrak n}}

\newcommand{\mk}{{\mathfrak k}}
\newcommand{\ma}{{\mathfrak a}}

\newcommand{\spl}{{\mathfrak sl}}
\newcommand{\spu}{{\mathfrak su}}

\newcommand{\q}{{\mathfrak q}}
\newcommand{\ml}{{\mathfrak l}}

\newcommand{\C}{{\bf C}}
\newcommand{\R}{{\bf R}}

\newcommand{\ad}{\rm ad}

\theoremstyle{plain} 
\newtheorem{Th}{\indent\bf Theorem}[section]
\newtheorem{Lm}{\indent\bf Lemma}[section]

\newtheorem{Ps}{\indent\bf Proposition}[section]

\theoremstyle{definition}

\newtheorem{Rm}{\indent\bf Remark}
\newtheorem{Ex}{\indent\bf Example}

\makeatother

\title{%\uppercase{
\huge Homogeneous Sasaki and Vaisman manifolds of unimodular Lie groups
}

\author{
\textsc{D. Alekseevsky, %V. Cor\'es, 
K. Hasegawa and Y. Kamishima}
}
\date{}

\begin{document}

\maketitle

%\footnote{ %2010MSC
%2010 \textit{Mathematics Subject Classification}.
%Primary 32M10, 53A30; Secondary 53B35.
%}
%\footnote{ %keywords
%\textit{Key words and phrases}.
%locally conformally K\"ahler structure, Sasaki structures, Vaisman structures,
%homogenous K\"ahler manifolds, unimodular Lie groups.
%}

\begin{abstract}
A Vaisman manifold is a special kind of locally conformally K\"ahler manifold, which is closely related
to a Sasaki manifold. In this paper we show a basic structure theorem of simply connected homogeneous Sasaki 
and Vaisman manifods of unimodular Lie groups, up to holomorphic isometry. For the case of unimodular Lie groups,
we obtain a complete classification of  simply connected Sasaki and Vaisman unimodular Lie groups, 
up to modification. 

\end{abstract}

%%%%%%%%%%%%%%%%%%%%%%%%%%%%%%%%%%%%%%%%%%%%%%%%%%%%%

\section*{Introduction} %%%%%%% %%%%%%%%%%%%%% introduction
%In our previous papers \cite{HK1}, \cite{HK2}, \cite{ACHK} we have discussed a basic framework of the structure of
%homogeneous locally conformally K\"ahler manifolds, and classified completely those of compact Lie group, up to
%holomorphic isometry, while showing that all of them are of Vaisman type. In this paper we extend our study to
%those of unimodular Lie group.
%\smallskip

We recall that a {\em locally conformally K\"ahler manifold}, or shortly an {\em l.c.K. manifold},
is an Hermitian manifold $(M, g, J)$,
where $g$ is an Hermitian metric with complex structure $J$ whose
associated fundamental $2$-form $\Omega$ satisfies the condition
$$d \Omega  = \Omega \wedge \theta \eqno(*)$$
for a closed $1$-form $\theta$, called the {\em Lee form}.
We may also define it as a {\em locally conformally symplectic manifold} with compatible complex structure
$(M, \Omega, J)$, where $\Omega$ is a non-degenerate $2$-form which satisfies $(*)$
for a closed $1$-form $\theta$ and $J$ is an integrable complex structure compatible with $\Omega$.
$M$ is of {\em Vaisman type} if the Lee form $\theta$ is parallel with respect to $g$, or equivalently 
if the {\em Lee field} $\xi$, the dual vector field of $\theta$ by the metric $g$,  is Killing.

A {\em homogeneous l.c.K. manifold} $(M, g, J)$ is a homogeneous Hermitian manifold whose
associated fundamental form $\Omega$ satisfies the above condition $(*)$;
in particular the Lee form $\theta$ is also invariant.
We can express $M$, if necessary, as $G/H$ in an effective form,
where $G$ is a connected Lie group of automorphisms of $(M, g, J)$, and $H$ is a closed subgroup of $G$
which does not contain any normal subgroup of $G$.
%A homogeneous Hermitian manifold $M=G/H$ is said to be of {\em unimodular type}, 
%if $G$ can be taken as a unimodular Lie group. 
\smallskip

Recall that a connected Lie group $G$ is {\em unimodular} if
it admits a bi-invariant Haar measure, or equivalently if its adjoint representation $\ad(X)$ 
in the Lie algebra $\g$ has trace zero for any $X \in \g$. Any compact, semi-simple, nilpotent,
reductive Lie groups, and Lie groups with a uniform lattice, are unimodular. Note that we have obtained
in the paper \cite{HK1} a complete classification of $4$-dimensional unimodular Lie groups with and
without lattices. 

\smallskip

A homogeneous l.c.K. manifold of a compact Lie group is nothing but
a compact homogeneous l.c.K. manifold; and we have already shown 
in \cite{HK2} a holomorphic structure theorem asserting that 
it is a holomorphic fiber bundle over a flag manifold with fiber a $1$-dimensional complex torus,
and a metric structure theorem asserting that all of them are of Vaisman type. 
Note that we have an extended version of the above metric theorem for homogeneous l.c.K. manifolds
in \cite{ACHK}, showing a sufficient condition for being of Vaisman type, 
that is, the normalizer of the isotropy subgroup $H$ in $G$ is compact,
while showing an example of a non-Vaisman l.c.K. structure on a reductive Lie group. 
For the $4$-dimensional case, we have shown in \cite{HK1} that a Hopf manifold of homogeneous type is
the only compact homogeneous l.c.K. manifold. 

\medskip

We recall that a {\em contact metric structure} 
$\{\phi, \eta, \widetilde{J}, g\}$ on $M^{2n+1}$
is a contact structure $\phi \, , \phi \wedge (d \phi)^n \not= 0$
with the {\em Reeb field} $\eta \, , i(\eta) \phi = 1, i(\eta) d \phi = 0$,
a $(1, 1)$-tensor $\widetilde{J} \, , \widetilde{J}^2 = -I + \phi \otimes \eta$
and a Riemannian metric $g$, satisfying $g(X, Y) = \phi(X) \phi(Y) + d \, \phi (X, \widetilde{J} Y)$.
A {\em Sasaki structure} on $M^{2n+1}$ is
a contact metric structure
$\{\phi, \eta, \widetilde{J}, g\}$ satisfying ${\cal L}_{\eta} g = 0$ (i.e. $\eta$ is a Killing field)
and the integrability of  $J = \widetilde{J}|{\cal D}$, where ${\cal D} = {\rm ker} \, \phi$ is a CR-structure.
%Equivalently, it is a contact metric structure satisfying
%$$(\nabla_X \widetilde{J}) Y= g(X,Y)\eta - \phi(Y) X \eqno(**)$$
%where $\nabla$ is the Riemannian connection of $g$. 

\smallskip

An automorphism of a Sasaki manifold $M$ is a diffeomorphism $\Psi$ which satisfies
$$\Psi^* \phi = \phi,\; J \Psi_* = \Psi_* J\,.$$ %\; \Psi_* {\cal D} \subset {\cal D}\,.$$
Note that the automorphism group of a Sasaki manifold is a closed 
Lie subgroup of the isometry group of $M$.
$M$ is a {\em homogeneous Sasaki manifold}, if a connected Lie group $G$ of automorphisms
acts transitively on $M$, that is, $M= G/H$ with isotropy subgroup $H$ of $G$. 
%A homogeneous Sasaki manifold $M=G/H$ is said to be of {\em unimodular type}, 
%if $G$ can be taken as a unimodular Lie group.

\medskip

%Note that ${\cal A} (M)$ is a closed Lie subgroup of the isometry group  ${\cal I} (M)$ of $M$;
%and it is compact if $M$ is compact.
%\smallskip

%and as an application of the classification of
%unimodular l.c.K.  Lie groups with and without lattices, we also have
%a complete classification of $4$-dimensional compact locally homogeneous l.c.K. manifolds

A {\em Sasaki (Vaisman) Lie group} $G$ is a homogeneous Sasaki
(Vaisman) manifold with trivial isotropy subgroup. We can define and study Sasaki (Vaisman) structures
on the Lie algebra $\g$ of $G$, which correspond uniquely to Sasaki (Vaisman) structures on $G$.
For l.c.K. structure on $\g$, we need only consider the structure $(g, J)$ or $(\Omega, J)$ on
$\g$ satisfying $(*)$ where $g$ is a Riemannian  metric and $\Omega$ is a non-degenerate $2$-form on $\g$
compatible with $J$.
Since the Lee form $\theta$ is closed, the Vaisman condition is just
$$ g([\xi, X], Y) + g(X, [\xi, Y]) =0,$$
that is, $\xi$ is Killing.

%For Sasaki structure on $\g$, we need only consider the contact metric structure $(\phi, \eta, \widetilde{J}, g)$
%on $\g$ satisfying $(**)$, where the Riemannian connection $\nabla$ can be defined using only
%the Riemannian metric $g$ on $\g$.
%We say that two l.c.K. manifolds are {\em isomorphic} if there is 
%a holomorphic isometry between them, and two Sasaki manifolds are
%{\em isomorphic} if there is an isometry preserving the Sasaki structures $(\phi, J, g)$ between them. 
\medskip

 A homogeneous Hermitian or Sasaki manifold may have different coset expressions
 $G/H$ and $G'/H'$. As a key strategy of proving a structure theorem of
homogeneous l.c.K. or Sasaki manifolds $G/H$ of a unimodular Lie group $G$, up to holomorphic isometry, 
we apply a {\em modification} of $G/H$ into $G'/H'$ (see Section $1$ for definition), 
which preserves holomorphic isometry and unimodularity. 
For Hermitian or Sasaki Lie groups, we see that modification is an equivalence relation, which preserves
Hermitian or Sasaki structures respectively and unimodularity.

\medskip

As main results of the paper we classify unimodular Sasaki and Vaisman Lie groups, up to modification 
(Theorem 2.1). More generally, we show a structure theorem for simply connected 
homogeneous Sasaki and Vaisman manifolds of unimodular Lie groups, up to holomorphic isometry (Theorem 4.1).

\medskip

{\em {\bf Theorem 2.1.} A simply connected Sasaki unimodular Lie group is
isomorphic to $N, SU(2)$, or $\widetilde{SL}(2, \R)$, up to modification.
Accordingly, a simply connected Vaisman unimodular Lie group is
isomorphic to one of the following, up to modification:
$$\R \times N, \R \times SU(2), \R \times \widetilde{SL}(2, \R),$$
where $N$ is a real Heisenberg Lie group and $\widetilde{SL}(2, \R)$ is the universal covering Lie group of $SL(2, \R)$.}

\bigskip
%\break

{\em {\bf Theorem 4.1.}  A simply connected homogeneous Vaisman manifold $M$ 
of a unimodular Lie group is holomorphically isometric to $M'=\R \times M_1$ with canonical
Vaisman structure, where $M_1$ is a simply connected 
homogeneous Sasaki manifold of unimodular Lie group, which is a quantization of a 
simply connected homogeneous K\"ahler manifold $M_2$ of a reductive Lie group. 
As an Hermitian manifold $M$ is a holomorphic principal bundle over a simply connected
homogeneous K\"ahler manifold $M_2$ with fiber $\C^1$ or $\C^*$.}

\smallskip

%\noindent
 In the above statement we mean by a {\em quantization} of a homogeneous K\"ahler manifold $M_2$,
a principal bundle $M_1$ over $M_2$ with fiber $\R$ or $S^1$ satisfying $ d \psi = \omega$ 
for a contact form $\psi$ on $M_1$ and the K\"ahler form $\omega$ on $M_2$.

\medskip

A basic idea of the proofs is to show first that, up to modifications, a simply connected
homogeneous Vaisman manifold of unimodular Lie group can be assumed to have
the form $M=G/H$, where $G$ is a simply connected unimodular Lie group
of the form $G=\R \times G_1$ and $H$ is a connected compact subgroup of $G_1$ 
with ${\rm dim}\,Z(G_1) = 1$, where $Z(G_1)$ denotes the center of $G_1$; and
combining with our previous results in \cite{HK2}, {\cite{ACHK}  yields the following
structure theorem:

\smallskip

{\em Let $\g, \h$ be the Lie algebras of
$G, H$ respectively. Then the pair $\{\g, \h\}$ is of the following form.
$$\g= \R \times \g_1,$$
where $\g_1= {\rm ker}\, \theta \supset \h$, and $\g_1$ can be expressed as a central extension of $\g_2$ by
$\R$:
$$0 \rightarrow \R \rightarrow \g_1 \rightarrow \g_2 \rightarrow 0.$$

The Lee field $\xi$ and the Reeb field $\eta =J\xi$ generate $Z(\g)$, where
$Z(\g)$ denotes the center of $\g$; and
the l.c.K. form $\Omega$ can be written as
$$\Omega= - \theta \wedge \psi + d \psi,$$
where $\psi$ is the Reeb form defining a contact form on the homogeneous Sasaki
manifold $M_1=G_1/H$. Let $\mk = \pi (\h)$ for the projection $\pi: \g_1 \rightarrow \g_2$.
Then the pair $\{\g_2, \mk\}$ defines a
homogeneous K\"ahler manifold $M_2=G_2/K$ with the K\"ahler form $\omega=d \psi| \g_2$,
where $G_1$ and $K$ are the Lie groups corresponding to $\g_1$ and $\mk$ respectively.}

\medskip

We further observe, applying some basic results from the field of homogeneous
K\"ahler manifolds, that  the homogeneous K\"ahler manifold $M_2$ associated to $\{\g_2, \mk\}$
is of reductive type. Hence we can reduce the classification problem of homogeneous
Vaisman manifolds of unimodular type to that of homogeneous Sasaki manifolds of the same type, 
which are quantizations of homogeneous K\"ahler manifolds of a reductive Lie group.
We already know that a simply connected homogeneous K\"ahler manifold of a reductive Lie group
is a K\"ahlerian product of $\C^k$ and a homogeneous K\"ahler manifold of semi-simple Lie group
(which has the structure of a holomorphic fiber bundle over a symmetric domain with fiber a flag manifold). 

Conversely, starting from a simply connected homogeneous K\"ahler manifold $M_2$
of a reductive Lie group, we may construct its quantization which will be a simply connected
homogeneous Sasaki manifold $M_1$ and then take a product with $\R$, making it a simply connected
homogeneous Vaisman manifold $M$ of unimodular type. Here the quantization must be
the one induced from a central extension of a K\"ahler algebra $(\g_2, \mk)$ of reductive
Lie algebra as stated above. We assert in general that a simply connected homogeneous 
K\"ahler manifold $M_2$ of a reductive Lie group is $\R$-quantizable to a simply connected 
homogeneous Sasaki manifold $M_1$ if and only if $M_2$ is a product of $\C^k$ and a symmetric domain, 
which is exactly the case when $M_2$ contains no flag manifolds; and $M_2$ is $S^1$-quantizable 
in all other cases.
%This may be considered as a partial extension of the known result that a compact
%homogeneous Sasaki manifold is a principal $S^1$-bundle over a flag manifold.

\medskip

The paper is organized as follows. In Section 1 we review some basic terminologies and results in the field of
homogeneous manifolds; in particular we discuss {\em Modification}, which was a key strategy 
in proving a structure theorem of homogeneous
K\"ahler manifolds (Fundamental Conjecture of Gindikin and Vinberg)
\cite{DN}, in a slightly more general setting. As an important observation we see that modification
in the category of unimodular Lie groups (Lie algebras) is an equivalence relation.
In Section 2 we discuss Sasaki and Vaisman Lie algebras (Lie groups); and prove Theorem 2.1.
In Section 3 we provide some examples of Vaisman and non-Vaisman l.c.K. Lie groups.
In Section 4 we discuss homogeneous Sasaki and Vaisman manifolds and prove Theorem 4.1;
and also prove, applying some results in the field of homogeneous K\"ahler manifolds, a more detailed
structure theorem of homogenous Sasaki manifolds of unimodular Lie groups (Theorem 4.2).

%\break
 %%%%%%%%%%%%%%%%%%%%%%%%%%%%%%%%%%%%%%%%%%%%%%%%%%%%
 
\section{Preliminaries} %%%%%%%%%%%%% Preliminaries
%\medskip

Let $M=G/H$ be a homogeneous space of a connected Lie group $G$ with closed
subgroup $H$. Then the tangent bundle of $M$ is given as a $G$-bundle 
$G \times_H \g/\h$ over $M=G/H$ with fiber $\g/\h$, where the action of $H$ on the
fiber is given by ${\rm Ad}(x) \, (x \in H)$. A vector field
on $M$ is a section of this bundle; and a $p$-form on $M$ is a section of the
$G$-bundle $G \times_H \wedge^p (\g/\h)^*$, where the action of $H$ on
the fiber is given by ${\rm Ad}(x)^* \, (x \in H)$. A vector field (respectively
$p$-form), which is invariant by the left action of $G$, is canonically identified
with an element of $(\g/\h)^H$ (respectively $(\wedge^p (\g/\h)^*)^H$),
which is the set of elements of $\g/\h$ (respectively $\wedge^p (\g/\h)^*$) invariant by
the adjoint action of $H$. An invariant complex structure $J$ on $M$ is likewise considered
as an element $J$ of ${\rm Aut}(\g/\h)$ such that $J^2=-1$ and
${\rm Ad}(x) J=J {\rm Ad}(x) \, (x \in H)$. Note that we may also consider an invariant
$p$-form as an element of $\wedge^p \g^*$ vanishing on
$\h$ and invariant by the action ${\rm Ad}(x)^* \, (x \in H)$; and an invariant complex structure
as an element $J$ of ${\rm End}(\g)$ such that $J^2 = - 1$ (mod $\h$) and $J \h \subset \h$.

\medskip

%%%%%%%%%%%%%%%%%%%%

We next define and discuss {\em modification} in the category of homogeneous Hermitian manifolds 
and Lie groups.
Let $\g$ be a Lie algebra with Hermitian structure $(g, J)$, and ${\rm Der} (\g)$
the derivation algebra of $\g$, which is a Lie subalgebra of ${\rm End} (\g)$.
Let $\mk$ be a subalgebra of ${\rm Der} (\g)$ consisting of
skew-symmetric derivations $\sigma$ compatible with $J$:
$$g(\sigma(X), Y)+g(X,\sigma(Y))=0,\,\,J \sigma = \sigma J \eqno(1.1)$$
for any $X, Y \in \g$.
We define the Lie algebra $\hat{\g}$ by setting
$$\hat{\g} = \g \rtimes \mk,$$
on which the new Lie brackets are defined by
$$[(X, \sigma), (Y, \sigma')] = ([X, Y] + \sigma(Y) - \sigma'(X), [\sigma, \sigma']).$$

We extend the metric $g$ and the complex structure $J$ to $\hat{\g}$,
setting $\hat{g}(\hat{\g}, \mk)=0$ and $\hat{J}(\mk)=0$.
We have a {\em modification} $\bar{\g}$ of $\g$:
$$\bar{\g}=\hat{\g}/\mk,$$
which is isomorphic to $\g$ as Hermitian vector space. Let $G$ (resp. $\hat{G}$) be
the simply connected Lie group with Lie algebra $\g$ (resp. $\hat{\g}$), and $K$ a compact
subgroup of $\hat{G}$ with Lie algebra $\mk$.
Then, $G$ is isomorphic to $\hat{G}/K$ as homogeneous Hermitian manifold. 
It should be noted that the unimodularity is preserved by the modification since it is a skew symmetric
operation (1.1).

\medskip

Any subgroup $G'$ of $\hat{G}$ canonically acts on $G$; and the action is simply transitive if and only if 
the Lie algebra $\g'$ of $G'$ is of the form
$$\g' = \{(X, \phi(X)) \in  \hat{\g} \,| \, X \in \g\},$$
where $\phi$ is a linear map from $\g$ to $\mk$. The Lie bracket on $\g'$ is defined by
$$ [(X, \phi(X)), (Y, \phi(Y))]' = ([X, Y] + \phi(X) (Y) - \phi(Y) (X),\, [\phi(X), \phi(Y)]),$$
In case $\mk$ is abelian, the projection of $\hat{\g}$ onto $\bar{\g}=\hat{\g}/\mk$ maps $\g'$ isomorphically
(as a vector space)
onto $\bar{\g}$, defining a Lie algebra structure on $\bar{\g}$:
$$[X, Y]^- = [X, Y] + \phi(X) (Y) - \phi(Y) (X) \eqno(1.2)$$
for $X, Y \in \bar{\g}$. The Lie group $\bar{G}$ with Lie algebra $\bar{\g}$ is called
a {\em modification} of the Lie group $G$. Note that $\bar{G}$ preserves the original Hermitian structure on $G$.

\medskip

We consider the set $\mathfrak L$ of linear maps $\phi: \g \rightarrow \mk$ satisfying the condition
$\phi([\g,\g])=0, \phi(\sigma(X))=0$ for any $X \in \g$ and $\sigma \in \mk$.
Since $\mk$ is abelian, $\mathfrak L$ may also considered as the set of Lie algebra homomorphisms 
$\phi: \g \rightarrow \mk$ satisfying the condition
$$\phi(\sigma(X))=0 \eqno(1.3)$$
for any $X \in \g$ and $\sigma \in \mk$. In particular, we have  $\phi_1(\phi_2(X)Y)=0$ for any $X, Y \in \g$ and
$\phi_1, \phi_2 \in {\mathfrak L}$.
It is easy to see that $\mathfrak L$ is a linear vector space, and any element $\phi(X)\; (X \in \bar{\g})$ defines
a skew symmetric derivation compatible with $J$ (condition (1.1)) with respect to
the new Lie bracket (1.2).
In particular, the modification of $\bar{\g}$ by $- \phi$ defines the original Lie algebra $\g$; and the composite
of two modifications $\phi_1, \phi_2$ is given by $\phi_1 + \phi_2$.
We also see that the modification is an equivalence relation in the set of Hermitian Lie algebras (groups).

\medskip

%%%%%%%%%%%%%%%%%%%%%%%%%%% example 1
\begin{Ex}

{\rm Let ${\g}'$ be a Lie algebra with a basis $\{X, Y, Z, W\}$ for which 
the bracket multiplication is defined by
$$[X,Y]=-Z, [W,X]=-Y, [W,Y]=X,$$
and other brackets vanish.
A complex structure $J$ on ${\g}'$ is defined by
$$JX=-Y, JY=X, JZ=-W, JW=Z \eqno(1.4)$$

An Hermitian metric $g$ is defined such that $X,Y,Z,W$ is an orthogonal basis.
Let $\n$ be the Heisenberg Lie algebra with a basis $\{X, Y, Z\}$ for which 
the bracket multiplication is defined by
$$[X,Y]=-Z,$$
and other brackets vanish.
We see that $\bar{\g}$ is a modification of  $\g=\n \times \R$.
A linear map $\phi: \g \rightarrow {\rm Der}\,(\g)$ is defined as
$$\phi(X)=\phi(Y)=\phi(Z)=0, \phi(W)={\rm ad}_W,$$
where ${\rm ad}_W$ is defined by
$${\rm ad}_W (X)=-Y, {\rm ad}_W (Y)=X, {\rm ad}_W (Z)=0, {\rm ad}_W (W)=0.$$

It is clear that ${\rm ad}_W$ is skew-symmetric with respect to $g$ and
compatible with $J$. Hence, setting $\mk =<{\rm ad}_W>$,  
we get a modification $\bar{\g}$ of $\g$:
$$\bar{\g}=\g \rtimes \mk/ \mk.$$
Since $\phi$ clearly satisfies the condition (1.3), 
$\bar{\g}$ can be identified with $\g'$ through the map 
$\psi: \g' \rightarrow \g \rtimes \mk  \rightarrow \bar{\g}$
defined by $\psi(X)={\rm pr} (X, \phi(X)) $.

\smallskip

Note that $\g$ is a nilpotent Lie algebra and $\g'$ is a unimodular non-nilpotent solvable Lie algebra.
The corresponding Lie groups $G$ and $G'$ with the integrable complex structure (1.4)
admit uniform lattices, defining Primary and Secondary Kodaira
surfaces respectively. Both of them are Vaisman Lie groups with a l.c.K. form $\Omega$ defined by
$$\Omega = x \wedge y + z \wedge w$$
with the Lee form $w$, where $x, y, z, w$ are the Maurer-Cartan forms corresponding to $X,Y,Z,W$ respectively.
}
\end{Ex}
%%%%%%%%%%%%%%%%%%%%%%%%%

\bigskip

We can define a modification of a pair $(\g, \h)$ of an Hermitian Lie algebra $\g$ and
a subalgebra $\h$ of $\g$ under the additional condition:
$$\sigma(\h) \subset \h,\, J \sigma = \sigma J \;({\rm mod}\; \h) \;\eqno(1.5)$$
for any $\sigma \in \mk$. We get a modification $(\g', \h')$ of $(\g, \h)$ as
$$\g' = \g \rtimes \mk, \,\h' = \h \rtimes \mk.$$

Let $G$  (resp. $G'$) be the simply connected
Lie group with Lie algebra $\g$ (resp. $\g'$), and $H$ (resp. $H'$) be its closed subgroup
with Lie algebra $\h$ (resp. $\h'$). $G'/H'$ is isomorphic to $G/H$ as Hermitian manifold:

\medskip

For modification in the category of homogeneous Sasaki manifolds $G/H$, or the corresponding
Lie algebras $(\g, \h)$ with Sasaki structure $\{\phi, \eta, \widetilde{J}, g\}$,
we consider a subalgebra $\mk$ of ${\rm Der} (\g)$ consisting of
skew-symmetric derivations $\sigma$ compatible with $\widetilde{J}$:
$$g(\sigma(X), Y)+g(X,\sigma(Y))=0,\,\,\widetilde{J} \sigma = \sigma \widetilde{J} \eqno(1.1)'$$
for any $X, Y \in \g$. Then we can define the modification of Sasaki algebras
$(\g, \h)$ in the same way as  for the case of Hermitian algebras.

\medskip

The following lemma is a key in the whole arguments for
the proofs of our main results.

%%%%%%%%%%%%%%%%%% lemma 1
\begin{Lm}

Let $M=G/H$ be a simply connected homogeneous Vaisman manifold
where $H$ is a connected subgroup of a simply connected Lie group $G$.
Then, we can modify, if necessary, $\g/\h$ into $\g'/\h'$ with ${\rm dim}\, Z(\g') = 2$
and ${\rm dim}\,\h' \leq {\rm dim}\, \h+1$. Hence, $G/H$ is isomorphic to $G'/H'$
as homogeneous Vaisman manifold, where $(G',H')$ is the corresponding Lie groups of $(\g', \h')$.
Similarly, for a simply connected homogeneous Sasaki manifold $G/H$ we can modify,
if necessary, $\g/\h$ into $\g'/\h'$ with ${\rm dim}\, Z(\g') = 1$.
\end{Lm}
%%%%%%%%%%%

\begin{proof}
In fact, the set of invariant vector fields can be identified
with $(\g/\h)^{\h}$; and since the Lee field $\xi$ and Reeb field $\eta= J \xi$
are invariant they belong to this set. Since $\xi$ and $\eta$ are Killing and
compatible with the complex structure $J$, they define ${\rm ad}_\xi$ and
${\rm ad}_\eta$ in ${\rm Der} (\g)$, which commute with each other and
are compatible with $J$. 
They are also ${\rm ad}\, h$-invariant for $h \in \h$.

Let $\mk=<{\rm ad}_\xi>$, and 
$\hat{\g}= \g \rtimes \mk, \hat{\h}=\h \times \mk$.
We have  $\g/\h=\hat{\g}/\hat{\h}$, where $\hat{\g}$ has a central 
element $\zeta=\xi - {\rm ad}_\xi$ in $\hat{\g}$ which is identified 
with $\xi \,({\rm mod}\, \hat{\h})$.
Since the Lee form $\theta$ is closed and $\theta(\xi)=1$, we have $\xi \not\in [\g,\g]$.
Hence we have a modification of $\g$ into $\hat{\g}/\mk = \g'$ and $\hat{\h}/\mk = \h'= \h$
through the map $X \rightarrow (X, \phi(X))$. Therefore we have
$$\g/\h=\g'/\h'=\g'/\h$$
with $\xi \in Z(\g')$. In particular we have $\g' = \R \times \g_1$ with $\g_1 = {\rm ker}\, \theta \supset \h$,
where $\R$ is generated by $\xi$.
%\medskip
Similarly, we can modify $\g'/\h'$
into $\g''/\h''$ with $\xi, \eta \in Z(\g'')$. Note that in case $\xi$ or $\eta$ is
already in $Z(\g)$, ${\rm ad}_\xi$ or ${\rm ad}_\eta$ 
is trivial; and thus $\g'= \g\ \times \mk, \h'=\h \times \mk$ 
without any modification on $\g$ and $\h$. 
Since for a homogeneous Vaisman manifold $G''/H''$, the dimension
of the center is not greater than $2$ (\cite{HK2}, \cite{ACHK}), the Lee field and the Reeb field generate
$Z(\g'')$. Since $\h'=\h$, we have ${\rm dim}\,\h'' \leq {\rm dim}\, \h+1$.
\qed
\end{proof}
%%%%%%%%%%%%%%

\medskip

We review some basic and historical results on a classification of homogeneous K\"ahler manifolds
(due to Dorfmeister, Nakajima, Vinberg, Gindikin, Piatetskii-Shapiro, Matsushima, Borel, Hano, Shima;
see \cite{B}, \cite{DN}, \cite{Han}, \cite{VGP} and references therein).
\smallskip

Let $M=G/H$ be a homogeneous K\"ahler manifold, where $H$ is a closed subgroup
of a simply connected Lie group $G$. Let $\g, \h$ be the Lie algebras of $G, H$ respectively.
Then, we can consider a K\"ahler structure on $G/H$ as a pair $(J, \omega)$ of a complex structure
$J \in {\rm End}(\g)$ and a skew symmetric
bilinear form $\omega$ on $\g$, satisfying the following condition:

%%%%%%%%%%%%%%%%%
\begin{list}{}{\topsep=5pt \leftmargin=20pt \itemindent=5pt \parsep=0pt \itemsep=3pt}

\item[ (i)] $J \, \h \subset \h$, $J^2 =-I \;({\rm mod}\, \h)$
\item[ (ii)] ${\rm ad}_X J = J \,{\rm ad}_X \;({\rm mod}\, \h)$ for $X \in \h$
\item[ (iii)] $[JX, JY] = [X,Y] + J \,[JX, Y] + J \,[X, JY] \;({\rm mod}\, \h)$
\item[ (iv)] $\omega(\h, \g)=0,\; \omega(JX,JY)=\omega(X,Y)$
\item[ (v)] $\omega([X,Y],Z) + \omega([Y,Z],X) + \omega([Z,X],Y) = 0$
\item[ (vi)] $\omega(JX, X) \not= 0$ for X $\not\in \h$
\end{list}
%%%%%%%%%%%%%%%%

A {\em K\"ahler algebra} $(\g, \h, J, \omega)$ is a Lie algebra $\g$ with subalgebra $\h$,
$J \in {\rm End}(\g)$ and a skew symmetric bilinear form $\omega$ on $\g$, satisfying the above conditions.
A {\em K\"ahler algebra} $(\g, \h, J, \omega)$ is {\em effective} if $\h$ includes no non-trivial
ideals of $\g$.
A {\em $J$-algebra} is a K\"ahler algebra $(\g, \h, J, \omega)$ with a linear form
$\rho$ such that $d \rho = \omega$. Note that the condition $d \rho = \omega$
is often referred to as {\em non-degenerate}; for a K\"ahler algebra of effective form,
it is actually equivalent to non-degeneracy of
the Ricci curvature form $\mathfrak r$ of the K\"ahler structure (due to Nakajima \cite{N}).

\medskip

{\bf Structure Theorem of Homogeneous K\"ahler Manifolds:} %%%%%%%%%%%%%%% theorem
{\em A homogeneous K\"ahler manifold is a holomorphic fiber bundle  over a homogeneous
bounded domain with fiber a K\"ahlerian product of a locally flat K\"ahler manifold and a flag manifold.
In particular, due to the Grauert-Oka principle \cite{G}, it is biholomorphic to the product of these complex manifolds. }

\medskip

A key idea of the proof \cite{DN} for the theorem is to show, applying modifications if necessary,
that there exists an abelian ideal $\ka$ and a $J$-algebra $\f$ containing $\h$ such that 
$$ \g = \ka \rtimes \f \eqno(1.6)$$
which is a semi-direct sum, and $\g$ is {\em quasi-normal}, that is, ${\rm ad} (X)$ has only 
real eigenvalues for any element $X \in {\rm rad} (\g)$, where ${\rm rad} (\g)$ is the radical of $\g$.
There also exists a compact  $J$-subalgebra $\q$ of $\f$ satisfying $\f \supset \q \supset \h$ 
for which we can express $M$ as a fiber bundle:

$$P/H \rightarrow M=G/H \rightarrow G/P \eqno(1.7)$$
where $P=AQ$ and $A, Q$ are the Lie groups associated to $\ka, \q$ respectively; and 
$P/H= A/\Gamma \times Q/H_0$ with $H=H_0 \Gamma$ for the connected component $H_0$ of $H$
and a discrete subgroup $\Gamma$ of $A$.
The base space $G/P$ defines a homogeneous bounded domain,
$A/\Gamma$ a locally flat complex manifold, $Q/H_0$ a flag manifold, and
the fibration is holomorphic. 

%%%%%%%%%%%%%%%%%%%%%%%%%%%%%%%%%%%%%%%%%%%%%%%%%%%%%%%%

\section{Sasaki and Vaisman Unimodular Lie algebras} %%%%%%%%%%% section 2

 A Lie group $G$ is a homogeneous space with its own transitive action on the left.
It is a homogeneous l.c.K. manifold if it admits a left-invariant Hermitian structure
$(g, J)$ satisfying
$$d \Omega = \Omega \wedge \theta$$
for its associated fundamental form $\Omega$ and a closed $1$-form $\theta$ (Lee form).
Note that $\theta$ must be also left-invariant. It is clear that $G$ admits a left-invariant
l.c.K. structure if and only if its Lie algebra $\g$ admits an l.c.K. form $\Omega$. We call
$\g$ with an l.c.K. form $\Omega$ an l.c.K. Lie algebra. 
%\smallskip

%We have already obtained in our previous papers \cite{HK1}, \cite{ACHK} a classification of
%l.c.K. reductive Lie algebras and nilpotent Lie algebras, determining at the same time
%which l.c.K. structures are of Vaisman type. In this section we determine all
%Vaisman unimodular Lie algebras, up to modifications. 

We already know  a classification of
l.c.K. reductive Lie algebras (\cite{HK1}, \cite{ACHK}), and l.c.K.  nilpotent Lie algebras (\cite {S1}, \cite{HK1}),
determining at the same time which l.c.K. structures are of Vaisman type. In this section we determine all
Vaisman unimodular Lie algebras, up to modifications.

%%%%%%%%%%%%%%%%%%%%%%%%%%%%%%%% theorem 1
\begin{Th} 

A Sasaki unimodular Lie algebra is, up to modification, isomorphic to one of the three types:
$\n$, $\spu(2)$, $\spl(2, \R)$. Accordingly, Vaisman unimodular Lie algebra is, up to modification,
isomorphic to one of the following:
$$\R \times \n, \R \times \spu(2), \R \times \spl(2, \R),$$
where $\n$ is a Heisenberg Lie algebra. In terms of Lie groups,
a simply connected Sasaki unimodular Lie group is, up to modification,
isomorphic to one of the three types: $N, SU(2), \widetilde{SL}(2, \R)$.
Accordingly, a simply connected Vaisman unimodular Lie group is, up to modification,
isomorphic to one of the following:
$$\R \times N, \R \times SU(2), \R \times \widetilde{SL}(2, \R).$$

\end{Th}
%%%%%%%%%%%%%%%%%%

\begin{proof}
Let $\g$ be a Vaisman unimodular Lie algebra of dimension $2k+2$
with an l.c.K. form $\Omega$ and Lee form $\theta$.
Applying modification, if necessary, we can assume that
$$\g=\R \times \g_0 \eqno(2.1)$$
where $\g_0 = {\rm ker}\,\theta$, and $\R$ is generated by the Lee field $\xi$.
$\g_0$ is a Sasaki Lie algebra with Reeb field $\eta$.
Let $\psi$ be the contact form and $\mk=<\eta>$, then $(\g_0, \mk, J|\g_0, d \psi)$
defines a K\"ahler algebra. The Ricci curvature form $\mathfrak r$ of the K\"ahler structure is given by
$${\mathfrak r}\,(X,Y)=-\kappa([X,Y]),$$
where $\kappa$ is the {\em Koszul form} defined by
$$\kappa(X)={\rm Tr}_{\g_0/\mk}\, ({\rm ad}\, JX - J{\rm ad}\, X),$$
which is well defined on $\g_0/\mk$ \cite{K}.
Now, in case ${\rm dim}\, Z(\g_0)= 1$,
$Z(\g_0)=\mk$, and $\g_0/\mk$ is a unimodular K\"ahler Lie algebra. Then due to Hano \cite{Han},
$\g_0/\mk$ is meta-abelian and locally flat; and thus, up to modification, isomorphic to $\C^n$ as K\"ahler algebra. 
Therefore we get $\g_0= \n$, up to modification. In case ${\rm dim}\, Z(\g_0)= 0$,
we see that the Ricci form $\kappa$ is non-degenerate. In fact, since we have 
 $i(\eta) \phi = 1, i(\eta) d \phi = 0$, and ${\rm ad}\,(\eta)$ is not trivial,
$\mk$ is not an ideal of $\g_0$. Therefore the K\"ahler algebra $(\g_0, \mk, d \psi)$ is in effective form.
Since the K\"ahler algebra $(\g_0, \mk, d \psi)$ is non-degenerate (that is, it defines a $J$-algebra)
the Ricci form  ${\mathfrak r}$ is non-degenerate \cite{N};
it follows (due to Hano \cite{Han}) that $\g_0$ must be semi-simple. Then it is well known \cite{BW} 
that $\g_0$ must be either $\spu(2)$ or $\spl(2, \R)$.
\qed
\end{proof}
%%%%%%%%%%%%%

%%%%%%%%%%%%%%%%%%%%%% Remark
\begin{Rm}

{\rm  A Vaisman unimodular solvable Lie algebra $\g$ is, up to modification, isomorphic to
$\R \times {\mathfrak n}$ (see Example 3.1. for a non-nilpotent case). Since modification $\phi$ is a skew-symmetric
operation, its eigenvalues are all pure-imaginary; in particular, a Vaisman unimodular completely
solvable Lie algebra is isomorphic to $\R \times {\mathfrak n}$ \,\cite{S2}.}

\end{Rm}

\begin{Rm}

{\rm We have determined all homogeneous l.c.K. structures on $\R \times \n$
and $\R \times \spu(2)$, which are all of Vaisman type \cite{HK1}. 
We have also determined all homogeneous l.c.K. structures on $\R \times \spl(2, \R)$,
some of them are of non-Vaisman type, as we will see in the next section.
}

\end{Rm}

%%%%%%%%%%%%%%%%%%%%%%%%%%%%%%%%%%%%%%%%%%%%%%%%%%%%%%%%%

\section{l.c.K. unimodular Lie groups of non-Vaisman type} %%%%%%%%%%%%% section 3

In this section we show examples of l.c.K. reductive Lie algebras of non-Vaisman type
(which we already discussed in our previous papers \cite{HK2}, \cite{ACHK}), illustrating
how Vaisman and non-Vaisman structures can be defined on $\R \times {\mathfrak sl}(2, \R)$.

%%%%%%%%%%%%%%%%%%%%%%%%%%%%% example 2
\begin{Ex}

{\rm There exists a homogeneous l.c.K. structure on $\g=\R \times {\mathfrak sl}(2, \R)$
which is not of Vaisman type. Take a basis $\{X, Y, Z\}$ for ${\mathfrak sl}(2, \R)$
with bracket multiplication defined by
$$[X,Y]=-Z,\, [Z,X]=Y,\, [Z,Y]=-X \eqno(3.1)$$
and $T$ as a generator of the center $\R$ of $\g$, where we set
$$ 
%W=\frac{1}{2}\left(
%\begin{array}{cc}
%1 & 0\\
%0 & 1
%\end{array}
%\right),
%\;
X=\frac{1}{2}\left(
\begin{array}{cc}
0 & 1\\
1 & 0
\end{array}
\right),
\;
Y=\frac{1}{2}\left(
\begin{array}{cc}
1 & 0\\
0 & -1
\end{array}
\right),
\;
Z=\frac{1}{2}\left(
\begin{array}{cc}
0 & 1\\
-1 & 0
\end{array}
\right).
$$

Let $t, x, y, z,$ be the Maurer-Cartan forms corresponding to $T, X,Y,Z$
respectively; then we have
$$d t=0, d x=z \wedge y, d y=x \wedge z, d z=x \wedge y \eqno(3.2)$$
and an l.c.K. structure $\Omega=z \wedge t + x \wedge y$
compatible with an integrable homogeneous complex
structure $J$ on $\g$ defined by
$$J Y=X, J X=-Y, J T=Z, J Z=-T.$$

We can generalize $\Omega$ to an l.c.K. structure of the form
$$\Omega_{\psi} = \psi \wedge t+ d \psi \eqno(3.3)$$
compatible with the above complex structure $J$ on $\g$,
where $\psi= a x + b y + c z$ with $a,b,c \in \R$.
%\smallskip

We see that the symmetric bilinear form $g_\psi(U,V)=\Omega_{\psi}(JU,V)$ is
represented, with respect to the basis $\{T, X, Y, Z\}$, by the  matrix 

$$ A= \left(
\begin{array}{cccc}
c & -b & a & 0\\
-b & c & 0 & a\\
a & 0 & c & b\\
0 & a & b & c
\end{array}
\right),$$
which has the characteristic polynomial $\Phi_A (u) = \{(u-c)^2 -(a^2+b^2)\}^2$,
and has only positive eigenvalues if and only if $c > 0, c^2 > a^2+b^2$. 
The Lee form is $\theta=t$ and the Lee field is
$$\xi=\frac{1}{D} (c T + b X - a Y),$$ 
with $D=c^2 -a^2-b^2$.
We have also
$$g_\psi(\xi,\xi)=\frac{c}{D}.$$
We can see that $g_\psi([\xi,U],V)+g_\psi(U,[\xi,V]) \not\equiv 0$ unless $a=b=0$.
In fact for $U=V=Z$, 
$$g_\psi([\xi, Z], Z)+g_\psi(Z,[\xi,Z])=2 g_\psi([\xi, Z], Z)= - \frac{2}{D}(a^2+b^2)=0$$
if and only if $a=b=0$. Conversely for the case $a=b=0$, it is easy to check that 
$g_\psi([\xi,U],V)+g_\psi(U,[\xi,V]) \equiv 0$. Therefore we have shown
%\medskip

{\em For $J$ and $\Omega_\psi$ defined above, $g_\psi$ defines a (positive definite) l.c.K. metric
if and only if  $c > 0, c^2 > a^2+b^2$. It is of Vaisman type if and only if
$ c >0,\, a=b=0$. And it is of non-Vaisman type if and only if $ c>0,\, c^2 > a^2 + b^2 > 0$.}

\medskip

%Note that for any lattice $\Gamma$ of $G = \R \times \widetilde{SL}(2, \R)$ with the above
%homogeneous l.c.K. structure, we get a complex surface $\Gamma \backslash G$
%(properly elliptic surface) with locally homogeneous non-Vaisman l.c.K. structure.
%\medskip
%\bigskip

We see that $\g$ can be modified into $\g' /<S>$,
where $\g' = \R \times {\mathfrak gl}(2, \R)$ for which the basis consists
of $X, Y, Z$ and
$$W=\frac{1}{2}\left(
\begin{array}{cc}
1 & 0\\
0 & 1
\end{array}
\right),
$$
and we set

$$S=\frac{1}{2}\left(
\begin{array}{cc}
1 & -1\\
1 & 1
\end{array}
\right).
$$

Since we have $W=Z+S \in {\mathfrak gl}(2, \R)$, ${\rm ad}_S$ defines a skew-symmetric action 
on $\g$ and $Z = W \;({\rm mod} S)$. Hence we get  $\g = \g' /<S>$ as an
l.c.K. algebra with the original l.c.K. form $\Omega$, which is of Vaisman type.
Note that ${\dim}_\R Z(\g')=2$. We see that for $\g$ with the l.c.K. form $\Omega_\psi$ of non-Vaisman 
type, ${\rm ad}_S$ is not compatible with the metric $g_\psi$. In fact for $U=bX-aY$, 
$$g_\psi([S, U], Z)+g_\psi(U,[S,Z])= g_\psi([Z, U], Z)= a^2+b^2=0$$
if and only if $a=b=0$. Hence we cannot modify $\g$ with the l.c.K. form $\Omega_\psi$ of non-Vaisman 
type into $\g = \g' /<S>$ with a compatible Vaisman structure.
}

\end{Ex} 

%%%%%%%%%%%%%%%%%%%%%%%%%%%%%%%%%%%%%%%%%%%%%%%%%%%%%%%%

\section{Homogeneous Sasaki and Vaisman manifolds of unimodular Lie group} %%%%%%%%%%% section 4

In this section we prove Theorem 4.1 and Theorem 4.2 as our main results.

\medskip

For any Sasaki manifold $N$, its {\em K\"ahler cone} $C(N)$ is defined as
$C(N)=\R_+ \times N$ with the K\"ahler form $\omega=r d r \wedge \psi + \frac{r^2}{2}d \psi$,
where a compatible complex structure $\widehat{J}$ is defined by 
$\widehat{J} \eta = \frac{1}{r} \partial_r$ and $\widehat{J} |_{\cal D} = J$.
Note that a contact metric manifold $N^{2n+1}$ with $\{\psi, \eta, \widetilde{J}\}$ is Sasaki
if and only if the K\"ahler cone $C(N)$ with $(\omega, \widetilde{J})$ is K\"ahlerian.

For any Sasaki manifold $N$ with contact form $\psi$,  we can define
an l.c.K. form $\Omega=\frac{2}{r^2} \omega = \frac{2}{r} d r \wedge \psi + d \psi$;
or taking $t= -2 \,{\rm log}\, r$,
$\Omega= - d t \wedge \psi + d \psi$ on $M=\R \times N$ or $S^1 \times N$ respectively,
which is of Vaisman type. We can define a family of complex structures $J$ compatible with
$\Omega$ by
$$J \,\partial_t = b \,\partial_t + (1+b^2) \,\eta, J \eta = - \partial_t - b \, \eta \eqno(4.1)$$
where $b \in \R$ and the Lee field is $J \eta$.
%Conversely, any simply connected complete Vaisman manifold
%is of the form $\R \times N$ with an l.c.K. structure as above,
%where $N$ is a simply connected complete Sasaki manifold.
We call $M$ a {\em canonical Vaisman manifold} associated to a Sasaki manifold $N$.

\medskip

We have shown in Lemma 1.1 that up to modifications, a simply connected homogeneous Vaisman manifold of
unimodular Lie group can be assumed to have the form $M=G/H$, where $G$ is a simply connected unimodular
Lie group of the form $G=\R \times G_1$ and $H$ is a connected compact subgroup of $G_1$ 
with ${\rm dim}\, Z(G_1)=1$; and combining with our previous results in \cite{HK2}, \cite{ACHK} yields 
%\smallskip

 %%%%%%%%%%%%%%%%%%%%%%%%%%%%%% proposition 1
\begin{Ps} 

Let $\g, \h$ be the Lie algebras of
$G, H$ respectively. Then the pair $\{\g, \h\}$ is of the following form.
$$\g= \R \times \g_1,$$
where $\g_1= {\rm ker}\, \theta \supset \h$, and $\g_1$ can be expressed as a central extension of 
a Lie algebra $\g_2$ by $\R$:
$$0 \rightarrow \R \rightarrow \g_1 \rightarrow \g_2 \rightarrow 0.$$
The Lee field $\xi$ and the Reeb field $\eta =J\xi$ generate $Z(\g)$; and
the l.c.K. form $\Omega$ can be written as
$$\Omega= - \theta \wedge \psi + d \psi,$$
where $\psi$ is the Reeb form defining a contact form on the homogeneous Sasaki
manifold $G_1/H$, $G_1$ being the simply connected unimodular Lie group corresponding to $\g_1$.
Let $\mk = \pi (\h)$ for the projection $\pi: \g_1 \rightarrow \g_2$.
Then the pair $\{\g_2, \mk\}$ defines a
homogeneous K\"ahler manifold $G_2/K$ with the K\"ahler form $\omega=d \psi| \g_2$,
where $G_2$ and $K$ are the Lie groups corresponding to $\g_2$ and $\mk$ respectively.

\end{Ps}

\medskip

Let $\g_1$ be the Sasaki algebra with the Reeb field $\eta$ and the K\"ahler form $\omega$ in Proposition 4.1.
Then, the Lie bracket on $\g_1$ is the extension of $\g_2$ given by
$$[X, Y]_{\g_1}=[X, Y]_{\g_2} - \omega(X,Y) \eta,\, [\eta, Z]_{\g_1}=0 \eqno(4.2)$$
for $X, Y, Z \in \g_2$. Conversely, given a K\"ahler algebra $\{\g_2, \mk\}$ with
a K\"ahler form $\omega$ we can define  a Sasaki Lie algebra $\g_1$, which is
a central extension with a generator $\eta$ of $\R$ by the above formula.
Since $\eta$ is Killing, $\g_1$ is unimodular if and only if $\g_2$ is unimodular.
Hence $M_1=G_1/H$ is of unimodular type if and only if
$M_2=G_2/K$ is of the same type.

\medskip

We have then the following, which is one of our main results:
%We state Theorem 4.1 and Theorem 4.2, and complete the proofs based on Proposition 4.1 and 
%the above arguments.

 %%%%%%%%%%%%%%%%%%%%%%%%%%%% theorem 2
\begin{Th}

A simply connected homogeneous Vaisman manifold $M$ 
of unimodular Lie group is holomorphically isometric to $M'=\R \times M_1$ with canonical
Vaisman structure, where $M_1$ is a simply connected 
homogeneous Sasaki manifold of unimodular Lie group, which is a quantization of a 
simply connected homogeneous K\"ahler manifold $M_2$ of reductive Lie group.
As an Hermitian manifold $M$ is a holomorphic principal bundle over a simply connected 
homogeneous K\"ahler manifold $M_2$ with fiber $\C^1$ or $\C^*$.

\end{Th}
%%%%%%%%%%%%%%%%%%%%%%%

For the proof of Theorem 4.1, since we have already discussed and proved the first part of the theorem,
we need only show the last part that {\em a simply connected 
homogeneous Sasaki manifold $M_1$ of unimodular Lie group has the structure as stated in the theorem},
which is covered and more detailed in the following theorem:
%Hence we have reduced the classification problem of homogeneous
%Vaisman manifolds of unimodular Lie groups to that of homogeneous Sasaki manifolds. 

%%%%%%%%%%%%%% theorem 3
\begin{Th} 

A simply connected homogeneous Sasaki manifold $M_1$ of unimodular Lie group is 
a quantization of a simply connected homogeneous K\"ahler manifold $M_2$ of reductive Lie group; that is,
$M_1$ is a principal bundle over $M_2$ with fiber $\R$ or $S^1$ satisfying $ d \psi = \omega$ 
for a contact form $\psi$ on $M_1$ and the K\"ahler form $\omega$ on $M_2$. 
 
\smallskip

The simply connected homogeneous K\"ahler manifold $M_2$ of reductive Lie group is
a K\"ahlerian product of $\C^k$, a flag manifold $Q/V$ with a compact semi-simple Lie group $Q$
and a parabolic subgroup $V$, and a homogeneous K\"ahler manifold $P/U$ with a non-compact 
semisimple Lie group $P$ and a closed subgroup $U$:
$$M_2 = \C^k \times Q/V \times P/U \eqno(4.3)$$
The homogeneous K\"ahler manifold $P/U$
has a structure of a holomorphic fiber bundle over a symmetric domain $P/L$ with fiber a flag manifold
$L/U$ for a maximal compact subgroup $L$ of $P$.

\smallskip

Furthermore, $M_1$ is $\R$-quantization of $M_2$ if and only if $M_2$ is a product of $\C^k$ and
a symmetric domain $P/L$ with $L=U$, and $S^1$-quantization of $M_2$ in all other cases.

\end{Th}
%%%%%%%%%%%%%%%%%%%%%%%
Note that a homogeneous Sasaki manifold, and more generally
a homogeneous contact manifold is necessarily regular (cf. \cite{BW}, \cite{HK2}).
Note also that Theorem 4.2 may be considered independently as a result on classification of 
homogeneous Sasaki manifolds of unimodular Lie groups, which extends a known result on
compact homogeneous Sasaki manifolds (cf. \cite{BG}). 

\medskip

First, we prove the following key result on homogeneous K\"ahler manifolds of unimodular Lie group,
which could be of independent interest.
%\medskip

%%%%%%%%%%%%%%%%%%%%%%%% proposition 2
\begin{Ps} 

A simply connected homogeneous K\"ahler manifold $M=G/K$ of unimodular Lie group $G$
is of reductive type; that is, the K\"ahler algebra $\{\g, \mk\}$ of $M$ has,
up to modification,  a decomposition
$$\g=\ka \rtimes \ml,$$
where $\ka$ is an abelian K\"ahler ideal of dimension $k$, $\ml$ is a semi-simple K\"ahler subalgebra 
which contains $\mk$. As a K\"ahler manifold, $M$ is a product of $\C^k$ and a homogeneous
K\"ahler manifold $N=L/K$ of a semi-simple Lie group $L$:
$$M = \C^k \times N.$$
Furthermore, $N$ can be decomposed into a K\"ahlerian product of flag manifolds and non-compact
homogeneous K\"ahler manifolds each of which is a holomorphic fiber bundle over
a symmetric domain with fiber a flag manifold.

\end{Ps}
%%%%%%%%%%%%%%%%%%%%%%

\begin{proof}
Let $M=G/K$ be a simply connected homogeneous K\"ahler manifold, where
$G$ is a unimodular Lie group and $K$ its closed subgroup. 
%We may assume that
%$G$ is a full holomorphic isometry group of $M$, and simply connected. 
We have a decomposition
$$\g=\ka \rtimes \f,$$
where $\ka$ is a maximal abelian $J$-ideal of $\g$ isomorphic to $\C^k$ 
and $\f$ is a $J$-subalgebra which contains $\mk$. 
Moreover, due to \cite{VGP}, $\f$ decomposes into a product of
a solvable $J$-subalgebra $\s$, a reductive $J$-subalgebra $\q$,
$$\f = \s \times \q,$$
where $\q$ contains $\mk$, and the center of $\q$  is contained in $\mk$.
%It is also known that $\s$ corresponds to a homogeneous domain, and 
We see, applying the De Rham decomposition of homogeneous K\"ahler manifolds (cf. \cite{KN}),
that $\s$ is actually the radical of $\f$, which is a maximal solvable ideal of $\f$. We see also that
$\ka \rtimes \s$ is the radical of $\g$. Since $\g$ is by assumption
a unimodular Lie algebra, so is $\ka \rtimes \s$. It follows, due to Hano \cite{Han} that $\s$ must be trivial.
Since the center of $\q$  is contained in $\mk$, we may express $\g$ as
$$\g = \ma \rtimes \ml,$$
where $\ml$ is the semi-simple part of $\q$ and $\mk$ is contained in $\ml$.
Since $M$ is by assumption simply connected, $\ma$ corresponds to $\C^k$ as a flat K\"ahler manifold,
and thus the action of $L$ (the Lie group corresponding to $\ml$)  on $\C^k$ is holomorphically isometric.
Thus as a K\"ahler manifold $M$ is isomorphic to $\C^k \times L/K$ (cf. \cite{DN}), where $L/K$ is a product of
homogeneous K\"ahler manifolds of compact semi-simple Lie groups and homogeneous K\"ahler manifolds
of non-compact semi-simple Lie groups each of which is a holomorphic fiber bundle 
over a symmetric domain with fiber a flag manifold (c.f. \cite{B}).
\qed
\end{proof} 
%%%%%%%%%%%%%%%%%%%%%%%%%%%%%%
%\medskip

Next we discuss a quantization of a homogeneous K\"ahler manifold $M_2=G_2/K$ of reductive type.
In case $M_2 = \C^k$, its quantization is the Heisenberg Lie group $N$, which is a central extension of 
$\R$ by $\C^k$. In case $M_2=L/K$ is a flag manifold, being a simply connected Hodge manifold, 
where $L$ is a compact semi-simple Lie group, it is quantizable to a compact 
simply connected homogeneous Sasaki manifold with fiber $S^1$. In case $L$ is a non-compact 
semi-simple Lie group, $M_2$ is a holomorphic fiber bundle over a symmetric domain $L/B$ 
with fiber a flag manifold $B/K$, where $B$ is a maximal compact Lie subgroup of $L$ containing $K$.
Since the flag manifold $B/K$ is a K\"ahler submanifold of $M_2=G/K$ and $S^1$-quantizable, $M_2$
is also $S^1$-quantizable. In general cases, for two or more homogeneous K\"ahler manifolds
each of which is quantizable, we can construct naturally a quantization of their products in the following way.
For two K\"ahler algebras $\g_2$ and $\g_2'$ with their central extension $\g_1$ and $\g_1'$ respectively,
we can define a new central $\R$-extension of $\g_2 \times \g_2'$ by taking $\R \times \R/\Delta = \R$
with $\Delta=\{(X, -X) | X \in \R \}$::
$$ 0 \rightarrow \R \rightarrow \g_1 \times_\Delta \g_1' \rightarrow \g_2 \times \g_2' \rightarrow 0 \eqno(4.4)$$
where  $\g_1 \times_\Delta \g_1' =  (\g_1 \times  \g_1')/\Delta$, the quotient Lie algebra by the canonical action 
of $\Delta$ on $\g_1 \times  \g_1'$. Correspondingly, we obtain a quantization $G_1 \times_\Delta G_1'$ of
$G_2 \times G_2'$; and in general the quantization $G_1/H \times_\Delta G_1'/H'$ of
$G_2/K \times G_2'/K'$. Now, in case $M_2$ is a product of $\C^k$ and a symmetric domain
(which is the case $B=K$),
since $M_2$ is contractible, it is $\R$-quantizable (but not $S^1$-quantizable) to a simply connected
homogeneous Sasaki manifold.
In all other cases, as we have seen in the above, since $L$ is a non-compact 
semi-simple Lie group and $K$ is not a maximal compact subgroup of $L$ (that is, $B \supsetneq K$),
$M_2$ is $S^1$-quantizable to a simply connected homogeneous Sasaki manifold.

\smallskip

This completes the proof of Theorem 4.2, and thus also Theorem 4.1.

\medskip
\noindent {\bf Acknowledgements.} The authors would like to thank Vicente Cort\'es for very useful comments
on an early version of the paper. The authors express special thanks to the anonymous referee
for carefully reading the draft, giving appropriate suggestions to improve readability of the paper.

%We have thus obtained the following result, the last part of Theorem 4.2, which could be of independent interest.

%%%%%%%%%%%%%%%%%%%%%%%% proposition 3
%\begin{Ps} 

%A simply connected homogeneous K\"ahler manifold of reductive Lie group is
%$\R^1$-quantizable or $S^1$-quantizable to a simply connected homogeneous Sasaki manifold, according to
%whether it is a product of $\C^k$ and a symmetric domain, or not. In other words it is  $\R^1$-quantizable 
%exactly when it contains no flag manifolds.

%\end{Ps}
%%%%%%%%%%%%%%%%%%%%%%%%%%

%\bigskip
%\bigskip
%%%%%%%%%%%%%%%%%%%%%%%%%%%%%%%%%%%%%%%%%%%%%%%%%%%%%%

%\break

%\makeaddress % To print address at the end. 
\end{document}